\documentclass[11pt]{article} 

\usepackage[utf8]{inputenc} 
\usepackage{amsthm}
\usepackage{amsmath,url}
\usepackage{amssymb}
\usepackage{mathtools}
\usepackage[numbers,sort]{natbib}
\DeclarePairedDelimiter{\ceil}{\lceil}{\rceil} 

\usepackage{geometry} 
\usepackage{graphicx} 

\usepackage{booktabs} 
\usepackage{array} 
\usepackage{paralist} 
\usepackage{verbatim} 
\usepackage{subfig} 
\usepackage{amsthm}
\usepackage{amsmath}

\usepackage{fancyhdr} 
\usepackage[nottoc,notlof,notlot]{tocbibind} 
\usepackage[titles,subfigure]{tocloft} 

\newtheorem{theorem}{Theorem}
\newtheorem{lemma}{Lemma}
\newtheorem{proposition}{Proposition}
\newtheorem*{claim*}{Claim}
\newtheorem{claim}{Claim}

\title{Intersecting longest paths in chordal graphs}

\author{Daniel J. Harvey
\and Michael S. Payne\thanks{
La Trobe University, Bendigo, Australia.  
}}

\begin{document}
\maketitle

\begin{abstract}
We consider the size of the smallest set of vertices required to intersect every longest path in a chordal graph. Such sets are known as longest path transversals. We show that if $\omega(G)$ is the clique number of a chordal graph $G$, then there is a transversal of order at most $4\ceil{\tfrac{\omega(G)}{5}}$. 
 We also consider the analogous question for longest cycles, and show that if $G$ is a 2-connected chordal graph then there is a transversal intersecting all longest cycles of order at most $2\ceil{\tfrac{\omega(G)}{3}}$. 

\end{abstract}

\section{Introduction} 

In this paper we consider the paths of maximum length among all paths in a graph. As longest paths represent the most indirect way one could travel through the graph from some place to another, we will adopt the terminology of Kapoor et al.~\cite{KapoorEtal,JobsonEtal} and refer to them as \emph{detours}. 
It is not too difficult to convince yourself that any two detours in a connected graph must intersect. In 1966 Gallai asked whether all detours in a connected graph share a common vertex~\cite{Erdosproc}. 
While the answer turned out to be negative for graphs in general, this question has led to a large number of positive results for particular graph classes.

In the negative direction, Walther gave a counterexample with 25 vertices soon after the question was posed~\cite{Walther}. The smallest possible counterexample is obtained from the Petersen graph by breaking one vertex into three degree one vertices~\cite{Zamfirescu76,WaltherVoss}. Thomassen showed there are infinitely many planar counterexamples (the `hypotraceable' planar graphs)~\cite{Thomassen76}.
The known counterexamples contain sets of at least seven detours that do not have a common intersection. It remains an open question whether every $k$ detours in any connected graph have a vertex in common, for $k = 3,4,5,6$~\cite{Zamfirescu01,Voss}.

On the other hand, there are a number of classes of connected graphs for which it has been proven that all detours intersect in some vertex. These include trees, cacti~\cite{KlavzarPet}, split graphs~\cite{KlavzarPet}, circular arc graphs~\cite{BalisterEtal,Joos}, outerplanar graphs~\cite{RezendeEtal13,Axenovich}, 2-trees~\cite{RezendeEtal13}, graphs with all non-trivial blocks Hamiltonian~\cite{RezendeEtal13}, partial 2-trees~\cite{ChenEtal17}, connected dually chordal graphs~\cite{JobsonEtal}, cographs~\cite{JobsonEtal} and more~\cite{CerioliLima20,GolanShan18,ChenFu}. The motivating question for the present research is whether chordal graphs (and thus all $k$-trees) can be added to this list~\cite{West}. Recall that a graph is \emph{chordal} if every cycle of four or more vertices contains a chord, or in other words, the only induced cycles are triangles.

Even if there is no single vertex intersecting every detour, we may ask for the smallest \emph{transversal} of vertices that intersects all detours~\cite{RautenbachSereni}. This is the approach we take in our study of chordal graphs, and 
the main result of this paper is the following, where $\omega(G)$ is the clique number.
\begin{theorem}
\label{theorem:chordal}
Given a connected chordal graph $G$, there exists a transversal of order at most $4\ceil{\tfrac{\omega(G)}{5}}$ that intersects every detour. Moreover, this transversal induces a clique in $G$.
\end{theorem}
%
The order of the smallest transversal intersecting all detours in a chordal graph was recently studied by Cerioli et al.~\cite{CerioliEtal20}.
They showed that there is a transversal of order at most $\max\{1,\omega(G)-2\}$.
Theorem~\ref{theorem:chordal} is an improvement on their result for $\omega = 15,19,20,23,24,25$ and $\omega \geq 27$.


The intersection properties of longest cycles in graphs have also been investigated, perhaps to a lesser extent. We also prove the following.
\begin{theorem}
\label{theorem:cycles}
Given a $2$-connected chordal graph $G$, there exists a transversal of order at most $2\ceil{\tfrac{\omega(G)}{3}}$ that intersects every longest cycle of $G$. Moreover, this transversal induces a clique in $G$.
\end{theorem}
For large $\omega$ this improves on recent results of Guti\'errez~\cite{Gutierrez}, who showed that a 2-connected chordal graph $G$ contains a transversal intersecting all longest cycles of order at most $\max\{1,\omega(G)-3\}$.
Thus Theorem~\ref{theorem:cycles} is an improvement for $\omega = 12$ and $\omega \geq 14$.


We also note the following result about the order of a transversal of detours in a planar graph. Cerioli et al.~\cite{CerioliEtal20} showed that, given a connected graph $G$ with treewidth $k$, there exists a transversal intersecting every detour of order at most $k$. Fomin and Thilikos~\cite{FominThil06} showed that a planar graph of order $n$ has treewidth at most $3.182\sqrt{n}$. Together these give the following result.  

\begin{theorem}
\label{theorem:planar}
Given a connected planar graph $G$ with $n$ vertices, there exists a transversal intersecting every detour of order at most $3.182\sqrt{n}$.
\end{theorem}
 This improves a previous result of Rautenbach and Sereni~\cite{RautenbachSereni}. 


\section{Preliminaries}
Given a graph $G$, a \emph{tree decomposition} of $G$ consists of a tree $T$ and a collection of bags $\mathcal{B}$ indexed by the nodes of $T$ such that:
\begin{itemize}
\item the bags of $\mathcal{B}$ contain vertices of $V(G)$,
\item for each $v \in V(G)$, the set of bags containing $v$ correspond to a non-empty connected subtree of $T$,
\item for each edge $vw \in E(G)$ there is at least one bag containing both $v$ and $w$.
\end{itemize}
The width of a tree-decomposition is the order of the largest bag, minus 1. The treewidth of a graph $G$ is the minimum width over all tree decompositions of $G$. Technically, we should distinguish a node of the tree $T$ and the bag of $\mathcal{B}$ indexed by the node, but in practice there is rarely any need, and as such we conflate the two.

The following key result about tree decompositions of chordal graph we shall use extensively: 
\begin{lemma}\label{lemma:tdcliques}
If $G$ is a chordal graph, then $G$ has a tree decomposition of minimum width in which every bag corresponds to a clique.
\end{lemma}
This result is well-known, and can be proven multiple ways. A result of Gavril showed that the chordal graphs are equivalent to the \emph{subtree graphs}: the intersection graph of subtrees of a tree. It is reasonably clear that the collection of subtrees corresponding to a chordal graph $G$ can be used to construct a tree decomposition for $G$ with the desired properties; we omit the full proof.

Now we consider the behaviour of the detours in a chordal graph. We start with an extremely straight-forward result which will be helpful later.
\begin{proposition}
\label{prop:extend}
Let $G$ be a connected chordal graph, $X$ a clique in $G$ and $P$ a detour in $G$. If $P$ has either an end vertex in $X$ or contains an edge of $X$, then $P$ must contain every vertex of $X$.
\end{proposition}
\begin{proof}
If $P$ has an end vertex in $X$ or contains an edge of $X$ but $v \in X$ is not in $P$, we could extend $P$ to contain $v$ either by adding an edge from the end vertex to $v$, or by replacing an edge of $P$ in $X$ with a 2-edge path containing $v$. Since $P$ is a detour, this is a contradiction.
\end{proof}

A \emph{clique-cut} in $G$ is a clique that is also a cut-set.
Suppose $P$ is a detour and $X$ is a clique such that $V(P) \cap X \neq \emptyset$. 
\begin{itemize}
\item If all the vertices of $V(P)-X$ are in a single component of $G-X$ then we say $P$ \emph{touches} $X$ (or is \emph{touching} with respect to $X$). Alternatively, if $X$ is a clique-cut and there are vertices of $V(P)-X$ in at least two components of $G-X$ then we say that $P$ \emph{crosses} $X$ (or is  \emph{crossing} with respect to $X$). 
\item If $P$ has both end vertices in the same component of $G-X$ then we say $P$ is \emph{closed} with respect to $X$, and we call the component of $G-X$ containing both end vertices the \emph{home component} of $P$ with respect to $X$. Alternatively, if the end vertices of $P$ are in different components we say $P$ is \emph{open} with respect to $X$.
\item If $|V(P) \cap X| \leq \ceil{\tfrac{\omega(G)}{5}}$ then we say $P$ is \emph{small} with respect to $X$, otherwise it is \emph{large}.
\end{itemize}
Often the clique $X$ will be clear from context, in these cases we will often drop ``with respect to $X$". We call a clique \emph{total} if every detour intersects $X$, otherwise it is \emph{non-total}. When $X$ is total, every detour is either small or large, and either closed touching, closed crossing or open crossing. It is not possible from the definition for a detour to be open and touching. 


Recall a \emph{separation} is a partition of the vertex set $(M,X,N)$ such that no edge exists with an end vertex in $M$ and in $N$. Note that if $B$ is a non-leaf bag in a tree decomposition, then $G-B$ is disconnected and $(M,B,N)$ is a separation where the components of $G-B$ are sorted into $M$ and $N$. Given a path $P$, let $f(P,M)$ denote the number of end vertices of the path in $M$ (and define $f(P,N)$ equivalently). Clearly $f(P,M),f(P,N) \in \{0,1,2\}$, and if $P$ has no end vertices in $X$, then $f(P,M)+f(P,N)=2$. 

\begin{lemma}
\label{lemma:paste}
Let $(M,X,N)$ be a separation in $G$ where $X$ is a cut-clique. If there exist two detours $P,Q$ in $G$ such that 
\begin{enumerate}
\item neither ends in $X$, 
\item both intersect both of $M$ and $N$, and
\item $f(P,M)=f(Q,M)$
\end{enumerate}
then $(P \cap X) \cap (Q \cap X) \neq \emptyset$.
\end{lemma}
\begin{proof}
Suppose otherwise for the sake of a contradiction, that is, there exist detours $P$ and $Q$ satisfying the three properties but also  $(P \cap X) \cap (Q \cap X) = \emptyset$. Note that since neither $P$ nor $Q$ end in $X$ and since every path has two ends, it follows $f(P,N)=f(Q,N)$ also. We construct a path $P_M$ as follows:
\begin{itemize}
\item Consider the graph $G-N$ and the path $P$ inside this graph. Since $P$ intersects $N$, $G-N$ will not contain all of $P$. It is possible that $P-N$ is a set of subpaths of $P$, and as such the number of end vertices of $P-N$ may not be two. However note that the number of end vertices of the subpaths in $P-N$ that are contained in $M$ is still $f(P,M)$; all other end vertices must be in $X$. Since $P$ enters $N$, there is at least one end vertex in $X$; if $f(P,M) \neq 1$ then there are at least two end vertices in $X$ (as exactly one end vertex only occurs when $P$ starts in $M$ and ends in $N$).
\item Since $X$ is a clique, we can join the subpaths of $P-N$ together using edges of $X$ to create one long path. Call this $P_M$. Note again that $f(P_M,M)=f(P,M)$. If $f(P,M)=0$ then $P_M$ has two end vertices in $X$, if $f(P,M)=1$ then $P_M$ has one end vertex in $X$, and if $f(P,M)=2$ then $P_M$ has no end vertices in $X$ but since $P$ enters $N$ at least one edge of $X$ is in $P_M$. 
\end{itemize}
Construct $P_N,Q_M$ and $Q_N$ equivalently. Note that since $P \cap Q \cap X = \emptyset$, it follows $P_M \cap Q_N = \emptyset$ and $Q_M \cap P_N = \emptyset$. Also note $f(P_M,M) + f(Q_N,N) = f(P,M) + f(Q,N) = f(P,M) + f(P,N) = 2$, and also $f(Q_M,M)+f(P_N,N) =2$.
We now construct a new path $I$. If $f(P_M,M)=f(Q_N,N)=1$, then create $I$ by linking the endvertex of $P_M$ in $X$ and the endvertex of $Q_N$ in $X$ via an edge of $X$. Since  $P_M \cap Q_N = \emptyset$, $I$ will be a path. Alternatively without loss of generality $f(P_M,M) = 0$ and $f(Q_N,N) = 2$. Now create $I$ by deleting an edge of $Q_N$ in $X$ and attaching these new end vertices to the two end vertices of $P_M$ in $X$. Again this will be a path. Also create a second path $I'$ in the same way using $Q_M$ and $P_N$. Now every vertex of $P$ and $Q$ is in $I \cup I'$. However, let $v$ be a vertex of $P \cap X$. The vertex $v$ is in $P_M$ and $P_N$ and so is in both $I$ and $I'$. Thus $|I| + |I'| > |P| + |Q|$. But since $P,Q$ have maximum length, $|I| + |I'| \leq |P|+|Q|$, a contradiction.
\end{proof}


\section{Proof of Theorem~\ref{theorem:chordal}}
We now provide a proof of Theorem~\ref{theorem:chordal}. This proof proceeds in several stages. First, we show that a connected chordal graph $G$ contains a special kind of clique, in which we shall find our traversal.

Note that by Lemma~\ref{lemma:tdcliques} and the Helly property of subtrees, there must exist at least one total clique in $G$.
We define a \emph{central clique} $X$ to be a total clique that satisfies at least one of the following properties:
\begin{enumerate}
\item $X$ is a total clique such that there is no detour $P$ which is small and touching with respect to $X$, but there is a small crossing detour,
\item $X$ is a total clique such that there exist two detours $P,Q$ that are both small and touching with respect to $X$, and have different home components,
\item $X$ is a total clique such that either $|X| \leq 4\ceil{\tfrac{\omega(G)}{5}}$ or there are no small detours at all with respect to $X$.
\end{enumerate}
We will use the above numbers to refer to the different types of central cliques. 

\begin{lemma}
\label{lemma:central}
A connected chordal graph $G$ has a central clique.
\end{lemma}
\begin{proof}
Suppose no central clique exists. 
Fix a tree decomposition $T$ of $G$ such that each bag is a clique and such that for any two adjacent bags $X$ and $Y$ neither $X \subseteq Y$ nor $Y \subseteq X$. (The first property follows from $G$ being a chordal graph, and the second is achieved by contracting an edge in the tree decomposition whenever the property is violated.) 
We construct a set of directed edges $D$ with one going out of each bag $B$ as follows:
\begin{itemize}
\item If $B$ is a total clique, then since $B$ is not a central clique, there must be at least one small touching detour, and all small touching detours have the same home component, which we call $C$. Direct the arc from $B$ towards the subtree of $T-B$ containing $C$.
\item If $B$ is a non-total clique, there is some detour $P$ that does not intersect $B$, and so there is a component $C$ of $G-B$ such that $P \subseteq C$. The component $C$ is entirely within a single subtree of $T-B$, and so we construct an arc from $B$ towards its neighbour in said subtree. Note that even if several detours do not intersect $B$, they must be inside the same component since any two detours intersect, and so our arc is well-defined.
\end{itemize}

Using the bags of $T$ as nodes and the directed edges $D$ we construct a directed graph where the underlying graph is acyclic and every node has outdegree exactly one. Thus there exists at least one 2-cycle.
%
%
%
%
Let $e$ be the underlying edge of the tree-decomposition corresponding to a $2$-cycle, and label the bags at the end vertices $B$ and $B'$. The edge $e$ splits $T$ into two sides, which we label left and right such that $B$ is on the left and $B'$ is on the right. 

There are now three possibilities to consider. Firstly, suppose that $B$ and $B'$ are both non-total cliques. So there exists a detour $P$ that does not intersect $B$ and is entirely on the right, and a detour $Q$ that does not intersect $B'$ and is entirely on the left. But then $P \cap Q = \emptyset$, which cannot occur as detours pairwise intersect. 

Secondly, suppose that $B$ and $B'$ are both total cliques. Let $P$ be a small touching detour with respect to $B$, which will have its home component to the right of $e$. Label the home component of $P$ by $C$. Let $Q$ be a small touching detour with respect to $B'$ with its home component to the left of $e$. Define $X=B \cap B'$. Suppose $P$ contains a vertex $v$ in $B-X$. Now every vertex of $P$ is contained within $B$ or $C$, but $v$ has no neighbours in $C$ as $v$ and $C$ are separated by $X$. Thus, since $v$ is incident to some edge in $P$, the vertex $u$ at the other end of this edge must be in $B$, and so all of $B$ is in $P$ by Proposition~\ref{prop:extend}. Since $P$ is small with respect to $B$, it follows $B$ is a central clique of type 3. Hence $P$ does not visit $B-X$, and similarly $Q$ does not visit $B'-X$. 
Thus both $P$ and $Q$ are small and touching with respect to the clique-cut $X$. The detours $P,Q$ still have different home components (to the right and left of $e$ respectively) and so $X$ is a central clique of type 2.

Thirdly, suppose (without loss of generality) that $B$ is total and $B'$ is non-total. Let $P$ be a small touching detour with respect to $B$, which thus has its home component to the right of $e$. Let $Q$ be a detour that does not intersect $B'$ (since it is non-total). Note $Q$ intersects $B$, and is entirely to the left of $e$. Again, let $X := B \cap B'$. As in the previous case, if $P$ contains a vertex $v$ in $B-X$, then all of $B$ is in $P$ and $B$ is a central clique of type 3. Hence we suppose that $P$ does not intersect $B-X$. However, $P$ and $Q$ must intersect, and $Q$ does not intersect $B'$ and therefore does not intersect $X$, which leaves nowhere for the two detours to intersect. This gives our final contradiction.
\end{proof}

We now let $X$ be a central clique as guaranteed by Lemma~\ref{lemma:central}. If $X$ is a central clique of type 3, then either $X$ itself or any subset of $X$ of size $4\ceil{\tfrac{\omega(G)}{5}}$ satisfies the requirements for Theorem~\ref{theorem:chordal}. So we assume that $X$ is a central clique of type 1 or type 2, and not type 3. 
%
The next step is to construct a special set of detours $\mathcal{F}$ which will help us find our transversal. We do so via the following algorithm:
\begin{enumerate}
\item Initialise $\mathcal{F} = \emptyset$.
\item Consider the small, closed, crossing detours with respect to $X$.
\begin{enumerate}
\item If no such detours exist, go to Step 3.
\item Otherwise, if all such detours have the same home component, add one of them to $\mathcal{F}$, then go to Step 3.
\item Otherwise, there exist small closed crossing detours with different home components. If there exists two such detours with different home components that do not intersect in $X$, add them both to $\mathcal{F}$. Go to Step 3.
\item Finally, there exist small closed crossing detours with different home components, but they all pairwise intersect in $X$. Add one such pair $P,Q$ to $\mathcal{F}$, then go to Step 3.
\end{enumerate}
\item Now consider small touching detours with respect to $X$. If none exist, go to Step 4, otherwise go to Step 5.
\item If there exists a small open crossing detour, add it to $\mathcal{F}$ and finish.
\item Add two small touching detours $P,Q$ with different home components to $\mathcal{F}$ and finish.
\end{enumerate}

Note that if the algorithm reaches Step 5, $X$ had at least one small touching detour (from the condition in Step 3). Thus $X$ was a central clique of type 2, so it has two small touching detours with different home components. Hence Step 5 is a valid operation. All other steps of the algorithm are self-evidently valid.

From the construction of $\mathcal{F}$ it is clear that every detour in $\mathcal{F}$ is small, and $\mathcal{F}$ contains at most four detours (e.g.~if Steps 2(c) and 5 operate). Also, no detour $P \in \mathcal{F}$ has an end vertex in $X$, otherwise by Proposition~\ref{prop:extend} all of $X$ is in $P$ and so, since $P$ is small, $X$ is a central clique of type 3. Initially define $F := \bigcup_{P \in \mathcal{F}} P \cap X$. Since $\mathcal{F}$ contains at most four small detours, it follows $|F| \leq 4\ceil{\tfrac{\omega(G)}{5}}$ and since $F \subseteq X$, $F$ is itself a clique of $G$. If $|F| < 4\ceil{\tfrac{\omega(G)}{5}}$, then add other vertices of $X$ to $F$ in order to force $|F| = 4\ceil{\tfrac{\omega(G)}{5}}$ while maintaining the fact that $F$ is a clique. Our final step is to show that $F$ is the transversal we require.

\begin{claim} 
$F$ is a transversal for the set of detours in $G$.
\end{claim}
\begin{proof}
Let $R$ be a detour in $G$. There are two cases to consider: either $R$ is large with respect to $X$ or $R$ is small with respect to $X$.
If $R$ is large with respect to $X$, then $|R \cap X| \geq \ceil{\tfrac{\omega(G)}{5}} + 1$. If $R \cap F = \emptyset$, then $$|X| \geq |F| + |R \cap X|\geq 4\ceil{\tfrac{\omega(G)}{5}} + \ceil{\tfrac{\omega(G)}{5}} + 1 \geq \omega(G) + 1,$$ which is a clear contradiction. Hence we may assume that $R$ is small with respect to $X$. We may also assume that $R \not\in \mathcal{F}$, since that case is clear. 
Suppose $R$ is a touching detour. Thus by Step 5, $\mathcal{F}$ contains two small touching detours $P,Q$ with different home components. Without loss of generality, suppose that the home component for $P$ is different to the home component for $R$. Since $P$ and $R$ are touching detours and need to intersect while having different home components, it follows they intersect in $X$ itself, and thus $R$ intersects $F$.
Hence we may assume that $R$ is a crossing detour. The first subcase we will now consider is the case when $R$ is open and no touching detours exist. Thus by Step 4 and the existence of $R$ there exists a small open crossing detour $P \in \mathcal{F}$. We construct a separation $(M,X,N)$ where the detours $P$ and $R$ both have one endvertex in $M$ and one in $N$. Thus by Lemma~\ref{lemma:paste} it follows $(P \cap X) \cap (R \cap X) \neq \emptyset$, and hence $R$ intersects $F$.

The next subcase is when $R$ is open and crossing but some touching detours exist. Then by Step 5 there exist two small touching detours $P,Q \in \mathcal{F}$ with different home components. Partition the components of $G-X$ into $M$ and $N$ such that the home component of $P$ is in $M$, the home component of $Q$ is in $N$ and both $M$ and $N$ contain one endvertex of $R$; since $R$ is open this partition is achievable. Let $R_M$ be the subpath of $R$ created by deleting $N$ and pasting the remaining subpaths of $R$ together using the edges of $X$. Define $R_N$ analogously, and without loss of generality we say $|R_M| > \tfrac{1}{2}|R|$. Let $Q'$ be a subpath of $Q$ with one endvertex in $N$, one endvertex in $X$ such that $|Q'| > \tfrac{1}{2}|Q|$. If $R_M$ and $Q'$ do not intersect in $X$ then they can be concatinated to create a longer detour, if they do intersect in $X$ then so do $R$ and $Q$, and so $R$ intersects $F$. These subcases account for when $R$ is open and crossing. 

Finally we need to account for when $R$ is closed and crossing. Since $R$ exists, Step 2 will add some detours to $\mathcal{F}$. First suppose that all of the small, closed crossing detours have the same home component, and let $P$ be the detour Step 2(b) added to $\mathcal{F}$. Now let $M$ be the home component of $P$ (and $R$), and let $N$ be all other components of $G-X$. Thus by Lemma~\ref{lemma:paste} applied to the separation $(M,X,N)$ it follows that $P$ and $R$ intersect in $X$, and thus $R$ intersects $F$. 

Hence we may assume there exists a pair of small closed crossing detours $P,Q$ in $\mathcal{F}$ with different home components. Suppose $P,Q$ do not intersect in $X$ (that is, Step 2(c) occurred). Clearly we may also assume that $R$ does not intersect either $P$ or $Q$ in $X$. Label the home components of $P,Q,R$ by $C_P,C_Q$ and $C_R$ respectively, and note $C_P \neq C_Q$. If $C_P = C_R$, then let $M = C_P$ and $N = G - X - C_P$ and apply Lemma~\ref{lemma:paste} to $(M,X,N)$, so that $P$ intersects $R$ in $X$. So we assume $C_P \neq C_R$ and similarly that $C_Q \neq C_R$, that is, that $C_P,C_Q,C_R$ are three distinct components. Let $M = C_P \cup C_Q$ and $N = G-X-(C_P \cup C_Q)$. If both $P$ and $Q$ intersect $N$ then by Lemma~\ref{lemma:paste} the paths $P$ and $Q$ must intersect in $X$, which they do not. Hence without loss of generality we suppose $P$ does not intersect $N$, that is, $P \subset C_P \cup C_Q \cup X$. Since $P$ is crossing it must intersect both $C_P$ and $C_Q$. By a similar argument one of $P$ and $R$ cannot intersect $G-X-(C_P \cup C_R)$; since $P$ intersects $C_Q$ it must be that $R \subset C_P \cup C_R \cup X$ and that $R$ intersects $C_P$ and $C_R$. Again, by a similar argument one of $Q$ and $R$ cannot intersect $G-X-(C_Q \cup C_R)$, and given previous results it must be $Q \subset C_Q \cup C_R \cup X$ and that $Q$ intersects $C_Q$ and $C_R$.

We will now argue in a similar fashion to Lemma~\ref{lemma:paste} to construct a longer detour. Consider the subpaths of $P \cap (C_P \cup X)$. All of the end vertices of these subpaths are in $X$ itself except for two inside $C_P$, and since $P$ crosses $X$ there are at least two subpaths in $P \cap (C_P \cup X)$. Create $P'$ by connecting the subpaths of $P \cap (C_P \cup X)$ using the edges of $X$; because $P \cap (C_P \cup X)$ has at least two subpaths it follows that $P'$ will contain at least one edge of $X$. Similarly, consider the subpaths of $P \cap (C_Q \cup X)$ and connect them up using edges of $X$ to create a path $P''$. Note that $P''$ will have both end vertices in $X$.

Construct $Q',Q'',R',R''$ in the same way (substituting different components as appropriate). Finally, construct the path $P^*$ from $P'$ and $Q''$ by removing an edge $e$ of $P'$ in $X$ and adding the edges from the two endvertices of $e$ to the endvertices of $Q''$. This will be a path since $P' \subset C_P \cup X$ and $Q'' \subset C_R \cup X$ and $P,Q$ do not intersect in $X$. Construct $Q^*$ from $Q'$ and $R''$ and $R^*$ from $R'$ and $P''$ in the same fashion. Every vertex of $P,Q,R$ is used at least once in $P^*,Q^*,R^*$, but the vertices of $P \cap X$ are used twice (in both $P^*$ and $R^*$). Hence $|P^*| + |Q^*| + |R^*| > |P| + |Q| + |R|$, but since $P,Q,R$ are detours, it follows at least one of $P^*,Q^*,R^*$ is longer than a detour, a contradiction. This accounts for the case when there exists a pair of small closed crossing detours with different home components that do not intersect in $X$.

The final subcase we must consider is when every pair of small closed crossing detours with different home components has an intersection between the detours in $X$. (That is, Step 2(d) occurred.) Let $P,Q$ denote the pair we added to $\mathcal{F}$. Now at least one of these detours has a different home component to $R$; without loss of generality we suppose it is $P$. Then $P$ and $R$ would have been a legitimate choice of detours to add to $\mathcal{F}$ instead of $P$ and $Q$. But we know that for each valid choice the detours intersect in $X$, that is, $R$ must intersect $P$ in $X$ and is therefore intersects $F$. This completes all possible cases.
\end{proof}

\section{Longest cycles in chordal graphs}

In this section we discuss how to apply the same ideas as those used in the proof of Theorem~\ref{theorem:chordal} to prove the following result for longest cycles in chordal graphs. 

\newtheorem*{repeatcycles}{Theorem \ref{theorem:cycles}}
\begin{repeatcycles}
Given a $2$-connected chordal graph $G$, there exists a transversal of order at most $2\ceil{\tfrac{\omega(G)}{3}}$ that intersects every longest cycle of $G$. Moreover, this transversal induces a clique in $G$.
\end{repeatcycles}

We will describe the modifications required to adapt the proof to longest cycles, rather than including a full proof of this result.
First note that longest cycles in $2$-connected graphs have the property that any two intersect~\cite{Gutierrez}, and that similarly to Proposition~\ref{prop:extend}, a longest cycle with an edge in a given clique must visit every vertex of that clique.

Cycles have no ends, so there is no need for the open/closed distinction used above. As before, given a clique-cut $X$ we say that a cycle $C$ \emph{touches} $X$ if $V(C) - X$ lies in a single component of $G-X$ (its home component), and otherwise it \emph{crosses} $X$.
We say a cycle $C$ is \emph{small} with respect to a clique $X$ if $|V(C) \cap X| \leq \ceil{\tfrac{\omega(G)}{3}}$.

Lemma~\ref{lemma:paste} can be replaced with a lemma that states that any two longest cycles crossing $X$ must intersect in $X$; otherwise the pieces of the cycles can be rearranged into two other cycles that also use some edges of $X$, contradicting the fact they were longest cycles.

For longest cycles, a \emph{central clique} is a total clique-cut that has either: a small crossing longest cycle but no small touching longest cycles; two small touching longest cycles with different home components; order at most $2\ceil{\tfrac{\omega(G)}{3}}$; or no small longest cycles at all.

To show there exists a central clique we argue as before. Assume there is none, and direct edges of the tree decomposition toward disjoint or small touching longest cycles. This means every bag has out degree one and so there is a directed $2$-cycle.
We consider the bags that constitute this 2-cycle, and distinguish cases by whether they are total clique-cuts or not. The argument proceeds exactly as in the proof of Lemma~\ref{lemma:central}.

Given a central clique $X$ we construct a transversal as follows. 
Firstly, if there is a small crossing longest cycle but no small touching longest cycles, we only need to take the vertices from one such longest cycle in $X$.
Secondly, if there are small touching longest cycles with different home components then we take the vertices of two such cycles $P$ and $Q$ in $X$.
Now any other touching longest cycle must hit $P$ or $Q$ in $X$ since it avoids one of the home components. 
If $R$ is a crossing longest cycle, at least half of $R$ is outside of one of the home components, say that of $P$. 
If $P$ and $R$ don't intersect in $X$ it is possible to take at least half of each of $R$ and $P$ and combine them into a longer cycle using edges of $X$.
Thirdly, if $X$ has order at most $2\ceil{\tfrac{\omega(G)}{3}}$ we take all of $X$. 
Finally, if $X$ is larger but has no small longest cycles at all, we take any $2\ceil{\tfrac{\omega(G)}{3}}$ vertices from $X$.
Thus we have a transversal of order at most $2\ceil{\tfrac{\omega(G)}{3}}$ as required.


\bibliography{detours}
\bibliographystyle{plain} 

\end{document}